\documentclass{amsart} 
\usepackage{graphicx}
\usepackage{hyperref}
\usepackage{amsmath, amsthm, amssymb, amscd}

\usepackage{mathrsfs}
\usepackage{mathtools}
\usepackage{enumerate}
\usepackage{hyperref}
\usepackage{verbatim}
\usepackage{centernot}
\usepackage{subcaption}
\usepackage{cite}
\usepackage{color}
\usepackage{esint}
\usepackage{tikz-cd}
\usetikzlibrary{shapes.geometric}
\usepackage[shortlabels]{enumitem}
 
\usepackage{bm}
\usepackage{bbm, dsfont}
\definecolor{mycolor}{rgb}{0.0, 0.75, 1.0}

\makeatletter
\newcommand{\labitem}[2]{%
\def\@itemlabel{\textbf{#1}}
\item
\def\@currentlabel{#1}\label{#2}}
\makeatother

\title[Dirichlet forms with prescribed Hölder regularity]{One-dimensional Dirichlet forms with prescribed Hölder regularity}
\author{Riku Anttila}
\address[Riku Anttila]{Department of Mathematics and Statistics, University of Jyväskylä, P.O. Box 35, FI-40014 Jyväskylä, Finland}
\email{riku.t.anttila@jyu.fi}

\author{Roope Anttila}
\address[Roope Anttila]{University of St Andrews, Mathematical Institute, St Andrews KY16 9SS, Scotland
}
\email{ra216@st-andrews.ac.uk}


\date{\today}
\thanks{Riku Anttila was supported by the Finnish Ministry of Education and Culture's Pilot for Doctoral Programmes (Pilot project Mathematics of Sensing, Imaging and Modelling) and Roope Anttila by EPSRC, grant no. EP/Z533440/1.}

\usepackage[shortlabels]{enumitem}
 
\newtheorem{theorem}{Theorem}[section]
\newtheorem{lemma}[theorem]{Lemma}
\newtheorem{proposition}[theorem]{Proposition}
\newtheorem{corollary}[theorem]{Corollary}
\newtheorem{example}[theorem]{Example}
\newtheorem*{theorem*}{Theorem}

\numberwithin{equation}{section}

\theoremstyle{definition}
\newtheorem{definition}[theorem]{Definition}

\theoremstyle{remark}
\newtheorem{remark}[equation]{Remark}
    \newcommand*{\N}{\mathbb{N}}
    \newcommand*{\R}{\mathbb{R}}
        \DeclareMathOperator{\supp}{supp}
        \DeclareMathOperator{\dimh}{dim_H}
        \DeclareMathOperator{\udimh}{\overline{dim}_H}
        \DeclareMathOperator{\ldimloc}{\underline{dim}_{loc}}

        \DeclarePairedDelimiter\abs{\lvert}{\rvert}
        \DeclarePairedDelimiter\norm{\lVert}{\rVert}
        \DeclareMathOperator*{\esssup}{ess\,sup}
    \usepackage{systeme}
    \makeatletter
    \let\c@equation\c@figure
    \makeatother
    
\begin{document}

\begin{abstract}
    We study a class of strongly local, regular Dirichlet forms on the standard unit interval. The aim of our work is to record some Hölder regularity properties that have not been noted in the prior literature. In particular, as our main result, we show that for every $\delta \in (0,1]$ there exists a metric measure space $(X,d,\mu)$ equipped with a strongly local, regular Dirichlet form $(\mathcal{E},\mathcal{F})$ on $L^2(\mu)$ with the property that $\delta$ is the supremum of $\alpha \in (0,1]$ for which the domain $\mathcal{F}$ of the Dirichlet form contains a non-constant $\alpha$-Hölder continuous function.
    To the best of our knowledge, such examples were previously known only for $\delta = 1$.
    In our construction, the value $ \delta$ is characterized by the upper Hausdorff dimension of a certain Radon measure that is used to define $(\mathcal{E},\mathcal{F})$.
\end{abstract}

\maketitle
\section{Introduction}
The general theory of \emph{Dirichlet forms} provides a far reaching generalization of the quadratic Dirichlet energy
\[
    f \mapsto \int_{\R^n} \norm{\nabla f}^2 \, \textup{d}x,
\]
see \cite{FOT,ChenFukushima} for the standard textbooks of Dirichlet forms, and Section \ref{sec:DF} for the definitions relevant to this article. Dirichlet forms provide a unified language to develop analysis of stochastic processes, Sobolev-type function spaces and regularity theory of partial differential equations, on metric spaces without any smooth structure.
For instance, there is a whole branch of analysis called \emph{analysis on fractals} which is largely dedicated to studying Dirichlet forms on self-similar sets, such as the Sierpi\'nski gasket, the Sierpi\'nski carpet and the Menger sponge \cite{AOF,DOF,BarlowPerkins88,BarlowBass89,BarlowBass99}.

Moreover, it is by now well-known that Dirichlet forms are related to another theory of analysis on metric spaces based on upper gradients, see \cite{KoskelaZhou,DF-PI,Gigli15} and also \cite[Section 14.3]{HKST}. 
On the other hand, we note that neither of these theories cover the other. For instance, upper gradients do not necessarily satisfy a parallelogram law, see \cite{Gigli15}.
See also \cite[Theorem 2.6]{Barlow-CRM13} for a proof showing that the theory of upper gradients degenerates on the Sierpi\'nski carpet. Furthermore, see \cite{BarlowBass89,KusuokaZhou92} for constructions of a strongly local, regular Dirichlet form on the Sierpi\'nski carpet and \cite{BarlowBass92,SC-uniqueness10,Hino05,Hino13} for some of its intriguing properties.

\subsection{Main result}\label{sec:MainResult}
The aim of this article is to improve our understanding on \emph{pointwise behavior} of functions in the domains of Dirichlet forms.
Specifically, we are interested in Hölder regularity, which often appears in the related literature on heat kernel estimates \cite{GHL15,GrigoryanTelcs}.
Let us introduce the necessary terminology required to state our results.

Given $\alpha \in (0,1]$ and a metric space $(X,d)$, a function $f : X \to \R$ is \emph{$\alpha$-Hölder continuous} (or $\alpha$-Hölder for short) if there is $K > 0$ such that
\[
    \abs{f(x)-f(y)} \leq Kd(x,y)^\alpha \quad \text{for all $x,y \in X$.}
\]
As usual, if $\alpha=1$, then $f$ is called a \emph{Lipschitz function}. Note that, assuming $(X,d)$ is geodesic, the above condition is superfluous for $\alpha \in (1,\infty)$; a differentiation shows that it leads to constant functions. For this reason, we only consider $\alpha \in (0,1]$.

Now, consider a metric measure space $(X,d,\mu)$ equipped with a strongly local, regular Dirichlet form $(\mathcal{E},\mathcal{F})$ on $L^2(\mu)$, see Section \ref{sec:DF} for the definitions. Throughout the work, we call the value
\[
  \delta_* := \sup \{ \alpha \in (0,1] \colon \text{there exists a non-constant $\alpha$-Hölder continuous $f \in \mathcal{F}$}  \}
\]
the \emph{critical Hölder exponent} of $(X,d,\mu,\mathcal{E},\mathcal{F})$.
Here we understand $\delta_* = 0$ if $\mathcal{F}$ contains no non-constant Hölder continuous functions. Note that $\alpha$-Hölder continuity implies local $\beta$-Hölder continuity for $\beta \in (0,\alpha]$ and in this sense, the precise value of $\delta_*$ captures the highest possible regularity carried by functions in the domain $\mathcal{F}$.
A classic illustrating example is the Rademacher's theorem which states that Lipschitz functions with $X = \R^n$ are differentiable Lebesgue almost everywhere; see the seminal work of Cheeger for an extension of Rademacher's theorem to a certain class of metric spaces \cite{cheeger}.

Our objective is to study the properties of the critical Hölder exponent $\delta_*$. To the best of our knowledge, there are no prior works that consider $\delta_*$ in the Dirichlet form literature, but some related questions have attracted a notable attention; see the discussion below.
We are interested, for instance, in the following questions.
\begin{enumerate}
    \item[(\textbf{Q1})] Which values can $\delta_*$ take?
    \item[(\textbf{Q2})] Is $\delta_*$ a strict supremum or a maximum?
    \item[(\textbf{Q3})] If $\alpha \in (0,\delta_*)$, are $\alpha$-Hölder functions dense in $\mathcal{F}$?
\end{enumerate}

While it is clear that the classical Dirichlet energy on $\R^n$ satisfies $\delta_*=1$, we are not aware of any $(X,d,\mu,\mathcal{E},\mathcal{F})$ for which the precise value of $\delta_*$ is known for $\delta_* \neq 1$. In other words, (\textbf{Q1}) seems to be wide open. Our main result is an answer to (\textbf{Q1}).

\begin{theorem}\label{thm-1}
    For every $\delta \in [0,1]$, there is a metric measure space $(X,d,\mu)$ and a strongly local, regular Dirichlet form $(\mathcal{E},\mathcal{F})$ on $L^2(\mu)$ so that the critical Hölder exponent of $(X,d,\mu,\mathcal{E},\mathcal{F})$ is equal to $\delta$.
\end{theorem}

This article also provides insight to (\textbf{Q2}) and (\textbf{Q3}), as well as discusses some other properties of $\delta_*$ through examples, see the end of Section \ref{sec:FractalGeom}. The remainder of this section overviews the critical Hölder exponent $\delta_*$ within the prior literature.

Even though the precise value of $\delta_*$ is unkown in most cases, there are many examples where it is known that $\delta_*>0$, or in other words, $\mathcal{F}$ contains non-constant Hölder continuous functions. This holds whenever $(X,d,\mu,\mathcal{E},\mathcal{F})$ satisfies \emph{sub-Gaussian heat kernel estimates}, see the introduction of \cite{KM-singularity20} for a nice overview on such examples in analysis on fractals.

In \cite{AEBSLaaksoSobolev2025}, it was verified that the domain of a self-similar Dirichlet form on the Laakso diamond space does not contain any non-constant Lipschitz functions. Note that this is not strong enough to preclude $\delta_* = 1$.

Lastly, we discuss the case of the self-similar Dirichlet form on the Sierpi\'nski gasket (with the Euclidean metric) which has received attention for some related questions. For this Dirichlet form, every $f \in \mathcal{F}$ is $\alpha$-Hölder continuous for $\alpha = \log (5/3)/\log 4$.
This can be understood as a counterpart of the Morrey inequality; see \cite[Corollary 2.1.17]{AOF} and \cite[Lemma 8.17]{DOF}.
In \cite[Section 6.3]{Murugan-unifdomain}, Murugan conjectured that $\mathcal{F}$ does not contain any non-constant Lipschitz functions; note that even if Murugan's prediction was correct this would not give a non-trivial upper bound for $\delta_*$. And lastly, there is an old conjecture of Barlow and Perkins from 1988 \cite[Section 9]{BarlowPerkins88} that the domain of the associated Laplace operator, let us denote it by $\mathcal{D}(\mathcal{A})$, does not contain any non-constant $\alpha$-Hölder continuous functions for $\alpha \in (\log(5/3)/\log 2,1]$. Moreover, they prove in \cite[Theorem 5.22]{BarlowPerkins88} that every $f \in \mathcal{D}(\mathcal{A})$ is $\alpha$-Hölder continuous for $\alpha = \log(5/3)/\log 2$. Hence, the conjecture of Barlow and Perkins concern a similar critical Hölder exponent for $\mathcal{D}(\mathcal{A})$, and claims it to be equal to $\log(5/3)/\log 2$.

The conjecture of Barlow and Perkins was positively resolved very recently by Qiu and Wang \cite{QiuWang26}. The question of Murugan, however, still remains open.

\subsection{Construction of Dirichlet forms}\label{intro-2}

In order to prove Theorem \ref{thm-1}, we consider a class of Dirichlet forms defined on the standard unit interval $[0,1]$. The idea is very simple, we just pull-back the classical Dirichlet energy on the real line along carefully chosen homeomorphsims.
More precisely, we fix a non-atomic Radon probability measure $\Lambda$ on $[0,1]$ with full topological support, denote its cumulative distribution function by $\mathscr{U}(x) := \Lambda([0,x])$, and define the Dirichlet form $(\mathcal{E},\mathcal{F})$ according to the formula
\[
    \mathcal{E}(f,g) := \int_0^1 \nabla(f \circ \mathscr{U}^{-1}) \cdot \nabla(g \circ \mathscr{U}^{-1})\, \textup{d}x. 
\]
Here $\nabla h$ denotes the usual weak derivative of an absolutely continuous function $h : [0,1] \to \R$.
The domain $\mathcal{F}$ consists of continuous functions $f : [0,1] \to \R$ so that $f \circ \mathscr{U}^{-1}$ is absolutely continuous and $\nabla (f \circ \mathscr{U}^{-1}) \in L^2(\textup{d}x)$ where $\textup{d}x$ denotes the Lebesgue measure restricted to $[0,1]$. Lastly, we also fix a reference measure $\mu$, which is any non-atomic Radon probability measure on $[0,1]$ with full topological support. The choice of $\mu$ will not be relevant in the Hölder continuity results but it is customary to make such a choice in the general theory of Dirichlet forms\footnote{A reader dedicated to not fixing $\mu$ could alternatively consider resistance forms \cite{AOF}.}.
In the following section, we sketch the main ideas to finding the critical Hölder exponent $\delta_*$ of $([0,1],\abs{\,\cdot\,},\mu,\mathcal{E},\mathcal{F})$.

Before getting into details about $\delta_*$ in this setting, let us briefly comment on the above construction. 
Firstly, a different formulation producing the above Dirichlet forms is introduced in 
\cite[Section 2.2.3]{ChenFukushima}; see also references in \cite[Section 3.5]{ChenFukushima} for some historically important works on the associated one-dimensional diffusion processes.
To conclude that these are indeed the same, one only needs to apply classical theory of functions of bounded variation.
The details are left to an interested reader, and the required tools to this end can be found, for instance, from Proposition \ref{prop:F-properties} and \cite[Page 112]{HKST}.
Secondly, we could have alternatively changed the metric on $[0,1]$ and kept $(\mathcal{E},\mathcal{F})$ as the standard Dirichlet energy to achieve the same effect.
The present formulation is chosen because it seems to lead to the cleanest exposition of our results. And lastly, we remark that an alternative definition $(f,g) \mapsto \int_0^1 \nabla f \cdot \nabla g \, \textup{d}\Lambda$ never leads to a fruitful theory when $\Lambda$ is not an absolutely continuous measure, see \cite[Theorem 3.1.6]{FOT} and \cite{MarinoLucicPasqualetto}.

\subsection{Fractal geometry}\label{sec:FractalGeom}
Given $([0,1],\abs{\,\cdot\,},d,\mu,\mathcal{E},\mathcal{F})$ as above, its critical Hölder exponent $\delta_*$ turns out to be related to the infinitesimal scaling properties of $\Lambda$.
Recall that a Radon measure $\nu$ on $[0,1]$ is \emph{$\alpha$-Frostman} if there is $C>0$ such that
\begin{equation*}
    \nu(B(x,r)) \leq Cr^\alpha \quad \text{for all $x\in[0,1]$ and $r>0$.}
\end{equation*}
Our work identifies the following weaker variant of the Frostman condition to be crucial. Due to the lack of better terminology, we call this condition the generalized $\alpha$-Frostman condition.

\begin{definition}\label{def:Frostman}
    For $\alpha \in (0,1]$, we call a Radon measure $\nu$ on $[0,1]$  \emph{generalized $\alpha$-Frostman} if there is a Borel set $E \subset [0,1]$ with $\nu(E) > 0$ satisfying
    \begin{equation}\label{eq:Frostman}
        \limsup_{r \downarrow 0} \frac{\nu(B(x,r))}{r^\alpha}
         < \infty \quad \text{for all $x \in E$.}
    \end{equation}
\end{definition}
    Stated in other words, $\nu$ is generalized $\alpha$-Frostman if and only if it contains a non-negligible piece where it satisfies the classical $\alpha$-Frostman condition, infinitesimally\footnote{Note that, by taking a subset of $E$ if necessary, the infinitesimal condition \eqref{eq:Frostman} can be upgraded to the classical Frostman condition on a set of positive measure.}.
    This notion has direct and well known connections to dimension theory of measures which is studied in \emph{fractal geometry}.
    Let us review the definitions relevant to the article, and refer to \cite{Falconer,Mattila} for a more comprehensive introduction to the subject. The \emph{lower pointwise dimension} of a Radon measure $\nu$ at $x \in [0,1]$ is
    \[
    \ldimloc(\nu,x):=\liminf_{r\downarrow0} \frac{\log \nu(B(x,r))}{\log r}.
    \]
    The \emph{upper Hausdorff dimension} of $\nu$ is then defined as the $\nu$-essential supremum
    \[
        \udimh\nu\coloneqq \esssup_{x\sim \nu}\ldimloc(\nu,x).
    \]
    The following lemma, which is standard in fractal geometry, records the relationship between the generalized $\alpha$-Frostman condition and the upper Hausdorff dimension.

\begin{lemma}\label{lemma:upperHausdorff}
    Let $\nu$ be a Radon measure on $[0,1]$.
    If $\alpha \in [0,\udimh\nu)$ then $\nu$ is generalized $\alpha$-Frostman. On the other hand, if $\alpha \in (\udimh\nu,1]$, then $\nu$ is not generalized $\alpha$-Frostman.
\end{lemma}

\begin{proof}
    This essentially follows by carefully comparing the relevant definitions to each other. Nevertheless, let us sketch the details since we do not expect the reader to have a background in fractal geometry.

    Assume first that $\alpha \in [0,\udimh\nu)$. It holds by the definition of $\udimh\nu$ that
    \[
        E := \{ x \in [0,1] \colon \ldimloc(\nu,x) > \alpha \}
    \]
    has $\nu(E)>0$. Then, by the definition of $\ldimloc(\nu,x)$, for every $x \in E$ there is $R_x \in (0,1)$ such that $\nu(B(x,r)) \leq r^\alpha$ for all $x \in E$ and $r \in (0,R_x)$. This proves that $\nu$ is generalized $\alpha$-Frostman.

    Assume next that $\alpha \in (\udimh\nu,1]$. Then there is $\varepsilon>0$ such that
    \[
        Z := \{ x \in [0,1] : \ldimloc(\nu,x) \leq \alpha - \varepsilon \}
    \]
    has $\nu([0,1] \setminus Z) = 0$. It is not difficult to see that \eqref{eq:Frostman} fails at every $x \in Z$. Hence, $\nu$ cannot be generalized $\alpha$-Frostman.
\end{proof}

These fractal geometric concepts give a complete characterization of the Hölder continuity properties of $(\mathcal{E},\mathcal{F})$.

\begin{theorem}\label{thm-2}
    Let $\Lambda$ and $(\mathcal{E},\mathcal{F})$ be as above, and let $\alpha \in (0,1]$. Then, there exists a non-constant $\alpha$-Hölder continuous function $f \in \mathcal{F}$ if and only if $\Lambda$ is generalized $\alpha$-Frostman.
\end{theorem}

Note that Lemma \ref{lemma:upperHausdorff} and Theorem \ref{thm-2} immediately yield a precise value for the critical exponent $\delta_*$ in our setting.
\begin{corollary}\label{cor:delta_*}
    If $\Lambda$ and $(\mathcal{E},\mathcal{F})$ are as above, then
    $\delta_*=\udimh\Lambda$.
\end{corollary}
Theorem \ref{thm-1} now follows once we find for every $\delta \in [0,1]$ a non-atomic Radon probability measure $\nu$ with full topological support satisfying $\udimh\nu=\delta$. Constructions producing such measures are widely known in the fractal geometry literature but we nevertheless give full details in Examples \ref{prop:not-gen-delta-frostman} and \ref{prop:0-dim-full-support}.

Furthermore, we discuss the question (\textbf{Q2}) by constructing examples that realize cases where $\delta_*$ is a strict supremum, and also cases where it is a maximum, see Examples \ref{prop:not-gen-delta-frostman} and \ref{prop:is-gen-delta-frostman}. The question (\textbf{Q3}) and some other density properties are studied in Section \ref{sec:density}.

The remainder of the introduction outlines some extensions of our framework.

\subsection{Heat kernel estimates}
For the sake of future reference, we also record results about heat kernel estimates in Section \ref{sec:analytic}. In particular, we show that given any $\beta > 2$ our construction realizes an example $([0,1],\abs{\, \cdot\,},\mu,\mathcal{E},\mathcal{F})$ that satisfies sub-Gaussian heat kernel estimates with space-time scaling $\Psi(r) = r^\beta$, see Theorem \ref{thm:inverse} and Remark \ref{rem:HKE} for details. To our knowledge, the existence of such an example  on the real line was not written down in the prior literature, although some experts were already aware of this phenomenon.\footnote{The first author learned about sub-Gaussian estimates on the real line for the first time from Mathav Murugan during an informal meeting in 2025. 
This discussion happened several months before the work related to this article started.} We note that the existence of a $(X,d,\mu,\mathcal{E},\mathcal{F})$ satisfying sub-Gaussian estimates with $\Psi(r) = r^\beta$ for $\beta > 2$ is well-known; see \cite{MuruganLaakso,Barlow-which}.

\subsection{Extension to $p$}
In the recent years, there has been an increasing interest in $p$-versions of Dirichlet forms, see \cite{EID2026,ResistanceConjecture,KajinoShimizu,Sasaya,kigami} and references therein. Our methods do not rely on the linearity of Dirichlet forms, and therefore our main results are formulated for general $p \in [1,\infty)$ instead of restricting to the case $p=2$.

\subsection*{Structure of the article}
In Sections \ref{sec:2} and \ref{sec:3} we set up the preliminaries and the general framework. Section \ref{sec:4} contains the proof of Theorem \ref{thm-2}.
In Section \ref{sec:examples} we complete the proof of Theorem \ref{thm-1} by giving suitable examples of measures and applying Theorem \ref{thm-2}.

Finally, in Section \ref{sec:analytic} we identify some conditions under which our construction exhibits finer analytic properties. These are unrelated to the Hölder continuity results but are of independent interest. For $p=2$, we obtain heat kernel estimates.

\section*{Acknowledgments}
During an informal meeting of the first author, Sylvester Eriksson-Bique and Mathav Murugan in 2025, Murugan explained a construction of a $([0,1],\abs{\,\cdot\,},\mu,\mathcal{E},\mathcal{F})$ that satisfies sub-Gaussian estimates with $\Psi(r) = r^\beta$ for any $\beta > 2$. At the time, we were unaware of sub-Gaussian estimates on the real line.
Section \ref{sec:analytic} was largely inspired by this discussion. We are deeply grateful to Eriksson-Bique and Murugan for the fruitful dialogue back then. We also thank Ryosuke Shimizu for pointing out some references related to this article, and Tuomas Orponen for kindly providing the example in Remark \ref{rem:Fail}.

\section{Preliminaries}\label{sec:2}

\subsection{Notation}
Throughout the article, $\mathcal{C}$ denotes the space of continuous functions $[0,1] \to \R$, and $\norm{\,\cdot\,}_\infty$ stands for its \emph{uniform norm}.  
The \emph{open balls} of $[0,1]$ are denoted by
\[
    B(x,r) := \{y \in [0,1] : \abs{x-y} < r\}.
\]
The notation $\textup{d}x$ always refers to the 1-dimensional Lebesgue measure on $[0,1]$.

When $\nu$ is a non-negative Radon measure on $[0,1]$ and $p \in [1,\infty)$, we use the usual notation $L^p(\nu)$ for the Lebesgue spaces, and denote the norm by $\norm{\,\cdot\,}_{L^p(\nu)}$.
If $\nu$ also has full topological support, meaning $\nu$ is positive on non-empty open sets, we always identify $\mathcal{C} \subset L^p(\nu)$.
Given a Borel measurable $f : [0,1] \to [0,1]$, we denote the \emph{push-forward} of $\nu$ along $f$ by $f_*(\nu)$. Recall that
\begin{equation}\label{eq:Push-forward}
    \int_{A} g \, \textup{d}f_*(\nu) = \int_{f^{-1}(A)} (g \circ f) \, \textup{d}\nu
\end{equation}
for all Borel measurable $g : [0,1] \to [-\infty,\infty]$ and all Borel sets $A \subset [0,1]$. For notational convenience, we shall write the integrals over intervals
\[
    \int_a^b f \,\textup{d}\nu := \int_{[a,b]} f \,\textup{d}\nu
\]
where $0 \leq a \leq b \leq 1$ and $f : [0,1]\to[-\infty,\infty]$ is Borel measurable.

\subsection{Sobolev spaces}
Let us next fix the terminology related to first-order Sobolev spaces on the real line. The relevant background may be found for instance in  Section 4 of the book by Evans and Gariepy \cite{Evans-Gariepy}.
For the basics of absolutely continuous functions see Chapter 7 of Rudin's book \cite{Rudin-2}. 

For $p \in [1,\infty)$, the notation $W^{1,p}$ always refers to the \emph{first-order $(1,p)$-Sobolev space} where the domain is the open interval $(0,1)$.
The \emph{weak derivative} of $h \in W^{1,p}$ on $(0,1)$ is denoted by $\nabla h \in L^p(\textup{d}x)$.
The usual Sobolev norm is denoted by $\norm{h}_{W^{1,p}} := \norm{h}_{L^p(\textup{d}x)} + \norm{\nabla h}_{L^p(\textup{d}x)}$. Lastly, whenever we write $h \in W^{1,p}$, we always identify it with the absolutely continuous representative $h : [0,1] \to \R$. Recall that this representative satisfies the \emph{fundamental theorem of calculus},
\begin{equation}\label{eq:abs-cont}
    h(y) = h(0) + \int_0^y \nabla h\, \textup{d}x \quad \text{for all $y \in [0,1]$,}
\end{equation}
see for instance the 5th exercise problem in \cite[Chapter 5]{evans}.
In particular, we understand that $W^{1,p} \subset \mathcal{C}$.

\section{The energy form and its basic properties}\label{sec:3}
This section provides the details about the construction discussed in introduction. As already mentioned, we consider general $p \in [1,\infty)$, but also discuss the $p=2$ case separately.

\subsection{General framework}\label{sec:Framework}
Fix an exponent $p \in [1,\infty)$ and a pair of Radon measures $\mu$ and $\Lambda$ on $[0,1]$. Both $\mu$ and $\Lambda$ are assumed to be probability measures of full topological support that contain no atoms.
The cumulative distribution function of $\Lambda$ is denoted by $\mathscr{U} : [0,1] \to [0,1]$ where
\[
    \mathscr{U}(x) := \Lambda([0,x])
\]
for all $x \in [0,1]$.
Note that $\mathscr{U}$ is an increasing homeomorphism, and that $\mathscr{U}_*(\Lambda) = \textup{d}x$. The former holds by the hypotheses. To justify the latter, first note that
\begin{align*}
    \Lambda\left( \mathscr{U}^{-1}((a,b)) \right) & =  \Lambda\left(\mathscr{U}^{-1}(a),\mathscr{U}^{-1}(b)\right)\\
    & = \Lambda\left( 0,\mathscr{U}^{-1}(b)\right) - \Lambda\left( 0,\mathscr{U}^{-1}(a)\right) = b-a,
\end{align*}
where $0 \leq a < b \leq 1$. The equality $\mathscr{U}_*(\Lambda) = \textup{d}x$ is now a consequence of Caratheodory's extension theorem.

Given these parameters, we define the $p$-energy form as follows.

\begin{definition}\label{def:p-EF}
    The \emph{$p$-energy form} is defined as the pair $(\mathcal{E},\mathcal{F})$ given by the following.
    First, the \emph{domain} $\mathcal{F} \subset \mathcal{C}$ consist of $f \in \mathcal{C}$ such that $f \circ \mathscr{U}^{-1} \in W^{1,p}$.
    The \emph{$p$-energy} $\mathcal{E} : \mathcal{F} \to [0,\infty)$ is given by $\mathcal{E}(f) := \int_0^1 \abs{\nabla (f \circ \mathscr{U}^{-1})}^p \, \textup{d}x$. The domain $\mathcal{F}$ is equipped with the \emph{Sobolev norm} $\norm{f}_{\mathcal{F}} := \norm{f}_{L^p(\mu)} + \mathcal{E}(f)^{\frac{1}{p}}$.
\end{definition}

Throughout, $(\mathcal{E},\mathcal{F})$ always refers to the $p$-energy form in Definition \ref{def:p-EF} given by the parameters that are clear from the context.
This admits an intrinsic first-order calculus. We define the \emph{derivation} $\textrm{D} : \mathcal{F} \to L^p(\Lambda)$ by
\[
\textrm{D} f :=  (\nabla (f \circ \mathscr{U}^{-1})) \circ \mathscr{U} \in L^p(\Lambda).
\]
By $\mathscr{U}_*(\Lambda) = \textup{d}x$ and \eqref{eq:Push-forward}, it holds that $F = H$ $\textup{d}x$-almost everywhere if and only if $F \circ \mathscr{U} = H \circ \mathscr{U}$ $\Lambda$-almost everywhere.
This shows that $\textrm{D} f$ does not depend on the choice of the $\textup{d}x$-representative of $\nabla (f \circ \mathscr{U}^{-1}) \in L^p(\textup{d}x)$.
Hence, the derivation $\textrm{D}$ is well-defined.
It also holds that
$\norm{\textrm{D} f}_{L^p(\Lambda)}^p = \mathcal{E}(f) < \infty$ for all $f \in \mathcal{F}$.

The following proposition shows that the natural map $f \mapsto f \circ \mathscr{U}^{-1}$ is an equivalence between the normed spaces $\mathcal{F}$ and $W^{1,p}$.
Note that this map is not always an isometry because the $L^p$-norms $\norm{f}_{L^p(\mu)}$ and $\norm{f \circ \mathscr{U}^{-1}}_{L^p(\textup{d}x)}$ are not comparable in general.
The equivalence of the Sobolev norms is a consequence of \eqref{eq:abs-cont}.

\begin{lemma}\label{lemma:comparability}
    The map $\mathcal{F} \to W^{1,p}$ given by $f \mapsto f \circ \mathscr{U}^{-1}$ is a linear bijection such that there is $C \geq 1$ satisfying
    \[
       C^{-1} \norm{f \circ \mathscr{U}^{-1}}_{W^{1,p}} \leq \norm{f}_{\mathcal{F}} \leq C \norm{f \circ \mathscr{U}^{-1}}_{W^{1,p}}
    \]
    for all $f \in \mathcal{F}$.
\end{lemma}

\begin{proof}
    Note that, by applying Hölder's inequality to \eqref{eq:abs-cont}, we get
    \begin{equation}\label{eq:Morrey}
        \abs{h(y) - h(z)} \leq \norm{\nabla h}_{L^p(\textup{d}x)}
    \end{equation}
    for all $h \in W^{1,p}$ and $y,z \in [0,1]$.

    Now, the map $f \mapsto f \circ \mathscr{U}^{-1}$ is a linear bijection $\mathcal{F} \to W^{1,p}$ by definition. 
    Thus, we need to check the comparability of the norms. Since $\mathcal{E}(f) = \norm{\nabla (f \circ \mathscr{U}^{-1})}_{L^p(\textup{d}x)}$, we only need to estimate the $L^p$-norms.
    Let $\nu := \mathscr{U}_*(\mu)$. By the formula \eqref{eq:Push-forward} and the fact that $\nu$ is a probability measure, it holds for all $f \in \mathcal{F}$ that
    \begin{align*}
        & \quad \, \int_0^1 \abs{f}^p \, \textup{d}\mu = \int_0^1 \abs{f \circ \mathscr{U}^{-1}}^p \, \textup{d}\nu\\
        \leq & \quad 2^{p-1} \left(\int_0^1 \int_{0}^1 \abs{(f \circ \mathscr{U}^{-1})(x) - (f \circ \mathscr{U}^{-1})(y)}^p\,\textup{d}x\, \textup{d}\nu(y) + \int_0^1 \abs{f \circ \mathscr{U}^{-1}}^p \,\textup{d}x\right) \\
        \leq & \quad 2^{p-1} \left( \int_0^1 \abs{\nabla(f \circ \mathscr{U}^{-1})}^p\, \textup{d}x + \int_0^1 \abs{f \circ \mathscr{U}^{-1}}^p\, \textup{d}x\right) = 2^{p-1} \norm{f \circ \mathscr{U}^{-1}}_{W^{1,p}}^p. 
    \end{align*}
    We also used $\abs{a+b}^p \leq 2^{p-1} (\abs{a}^p + \abs{b}^p)$ and Jensen's inequality in the second line, and \eqref{eq:Morrey} in the third.
    This proves the latter inequality in the claim.
    The former follows from a similar computation
    \begin{align*}
        & \quad \, \int_0^1 \abs{f \circ \mathscr{U}^{-1}}^p \, \textup{d}x\\
        \leq & \quad 2^{p-1} \left(\int_0^1\int_0^1 \abs{f \circ \mathscr{U}^{-1}(x) - f \circ \mathscr{U}^{-1}(y)}^p \, \textup{d}x\,\textup{d}\nu(y) + \int_0^1 \abs{f \circ \mathscr{U}^{-1}}^p \, \textup{d}\nu\right) \\
        \leq & \quad 2^{p-1} \left(\int_0^1 \abs{\nabla(f \circ \mathscr{U}^{-1})}^p\, \textup{d}x + \int_0^1 \abs{f}^p\, \textup{d}\mu\right) = 2^{p-1}\norm{f}_{\mathcal{F}}^p.
    \end{align*}
\end{proof}

The following properties of $(\mathcal{E},\mathcal{F})$ are needed in our proofs.

\begin{proposition}\label{prop:F-properties}
    The following conditions hold.
    \begin{enumerate}
        \item $(\mathcal{F},\norm{\,\cdot\,}_\mathcal{F})$ is a Banach space.
        \item Let $f \in \mathcal{F}$ and $A \subset [0,1]$ be a Borel set so that $f$ is constant on $A$. Then $\textup{D}f(x) = 0$ for $\Lambda$-almost every $x \in A$.
        \item Let $f \in \mathcal{F}$. Then it holds for all $0 \leq z \leq y \leq 1$ that
        \[
           f(y) - f(z) = \int_z^y \textup{D}f \, \textup{d}\Lambda.
        \]
        In particular, if $\norm{\textup{D}f}_{L^p(\Lambda)}=0$, then $f$ is constant on $[0,1]$.
        \item Let $f \in \mathcal{C}$ so that there is $G \in L^p(\Lambda)$ satisfying
        \[
            f(y) = f(0) + \int_0^y G \, \textup{d}\Lambda
        \]
        for all $y \in [0,1]$. Then $f \in \mathcal{F}$ and $\textup{D}f = G$.
        \item If $f,g \in \mathcal{F}$ then $f \cdot g \in \mathcal{F}$ and
        \[
            \textup{D} (f\cdot g) = \textup{D}f \cdot g + f \cdot \textup{D}g.
        \]
    \end{enumerate}
\end{proposition}

\begin{proof}
    Statement (1) follows from Lemma \ref{lemma:comparability} because $W^{1,p}$ is a Banach space; see Theorem 2 of \cite[Chapter 5]{evans}.

    For (2), note that if $f$ is constant on $A$, then $f \circ \mathscr{U}^{-1}$ is constant on $\mathscr{U}(A)$. By \cite[Theorem 4.4-(iv)]{Evans-Gariepy}, this implies that $\nabla (f\circ\mathscr{U}^{-1})=0$ $\textup{d}x$-almost everywhere on $\mathscr{U}(A)$, which gives the claim since $\mathscr{U}_*(\Lambda)=\textup{d}x$.

    Statement (3) follows by simply computing
    \begin{align*}
        f(y) - f(z) & = (f \circ \mathscr{U}^{-1})(\mathscr{U}(y)) - (f \circ \mathscr{U}^{-1})(\mathscr{U}(z))\\
        & = \int_{\mathscr{U}(z)}^{\mathscr{U}(y)} \nabla (f \circ \mathscr{U}^{-1})\, \textup{d}x = \int_{z}^{y} \textup{D}f\, \textup{d}\Lambda.
    \end{align*}
    In the second row, we first used \eqref{eq:abs-cont}, and then applied \eqref{eq:Push-forward} together with the equality $\mathscr{U}_*(\Lambda) = \textup{d}x$.

    If $f$ and $G$ are as in statement (4), it follows that
    \[
        (f \circ \mathscr{U}^{-1})(y) = f(0) + \int_0^{\mathscr{U}^{-1}(y)} G \, \textup{d}\Lambda = f(0) + \int_0^{y} G \circ \mathscr{U}^{-1} \, \textup{d}x.
    \]
    We used $\mathscr{U}_*(\Lambda)= \textup{d}x$ in the last equality. In particular, $(f \circ \mathscr{U}^{-1}) : [0,1]\to \R$ is absolutely continuous with $\nabla (f \circ \mathscr{U}^{-1}) = G \circ \mathscr{U}^{-1} \in L^p(\textup{d}x)$. Hence, the claim follows from Definition \ref{def:p-EF} and the definition of $\textup{D}$.

    Finally, to prove (5), recall first that the space of bounded Sobolev functions is closed under multiplication, see \cite[Theorem 4.4-(i)]{Evans-Gariepy}. From this, it easily follows that $\mathcal{F}$ is also closed under multiplication. Moreover, by the usual Leibniz rule of Sobolev functions \cite[Theorem 4.4-(ii)]{Evans-Gariepy} and the definition of the derivation $\textup{D}$, it holds for all $f,g \in \mathcal{F}$ that
    \begin{align*}
        \textup{D}(f \cdot g) & = (\nabla [( f \circ \mathscr{U}^{-1} ) \cdot (g \circ \mathscr{U}^{-1} ] ) \circ \mathscr{U}\\
        & = (\nabla(f \circ \mathscr{U}^{-1}) \cdot (g \circ \mathscr{U}^{-1} ) +  (f \circ \mathscr{U}^{-1} )\cdot \nabla(g \circ \mathscr{U}^{-1} )) \circ \mathscr{U}\\
        & = \textup{D}f \cdot g + f \cdot \textup{D}g.
    \end{align*}
\end{proof}

\subsection{Dirichlet form}\label{sec:DF}
We finish the section by verifying that for $p=2$ our construction produces Dirichlet forms. Let us briefly recall the definitions relevant to the article; see for instance \cite[Section 2]{KM-singularity20} for further details and \cite{FOT,ChenFukushima} for the general theory.

Let $(X,d,\mu)$ be a locally compact, separable metric measure space. Denote by $L^2(\mu)$ the Lebesgue space, and by $C_c(X)$ the set of compactly supported continuous functions $X \to \R$. A \emph{Dirichlet form} on $L^2(\mu)$ is a pair $(\mathcal{E},\mathcal{F})$ consisting of a dense linear subspace $\mathcal{F}\subset L^2(\mu)$ and a non-negative definite bilinear form $\mathcal{E}\colon \mathcal{F}\times \mathcal{F}\to \R$, which is closed and Markovian. Being \emph{closed} means that $\mathcal{F}$ equipped with the inner product $(f,g)\mapsto \int f\cdot g\,\textup{d}\mu+\mathcal{E}(f,g)$ is a Hilbert space, and being \emph{Markovian} is a requirement that $f\in \mathcal{F}$ implies $\tilde{f}\in \mathcal{F}$ and $\mathcal{E}(\tilde{f},\tilde{f})\leq \mathcal{E}(f,f)$ where $\tilde{f}=\min\{\max\{f,0\},1\}$.
Furthermore,  $(\mathcal{E},\mathcal{F})$ is called \emph{regular} if $\mathcal{F}\cap C_c(X)$ is dense in both $(C_c(X),\norm{\,\cdot\,}_{\infty})$ and $(\mathcal{F},\norm{\,\cdot\,}_{\mathcal{F}})$, where $\norm{\,\cdot\,}_\infty$ is the uniform norm and $\|f\|_{\mathcal{F}}:=\|f\|_{L^2(\mu)}+\sqrt{\mathcal{E}(f,f)}$ for $f\in\mathcal{F}$. Finally, $(\mathcal{E},\mathcal{F})$ is  \emph{strongly local}, if $\mathcal{E}(f,g)=0$ when $f,g\in\mathcal{F}$ with $\supp_{\mu}[f]$ and $\supp_{\mu}[g]$ compact, and $f$ is constant $\mu$-almost everywhere on an open neighborhood of $\supp_{\mu}[g]$. Here $\supp_{\mu}[f]$ denotes the $\mu$-essential support of $f$ which is the smallest closed set $F\subset X$ satisfying $\int_{X\setminus F}|f|\,\textup{d}\mu=0$.

Let us return back to the setting of this article and define the two variable form $\mathcal{E}\colon \mathcal{F}\times \mathcal{F}\to \R$ by 
\[
    \mathcal{E}(f,g) := \int_0^1 \nabla(f \circ \mathscr{U}^{-1}) \cdot (g \circ \mathscr{U}^{-1})\, \textup{d}x = \int_0^1 \textrm{D} f \cdot \textrm{D} g\, \textup{d}\Lambda.
\]
Note that, in the notation of the previous subsection, $\mathcal{E}(f,f) = \mathcal{E}(f)$ when $p=2$. The following result, which is essentially a corollary of Proposition \ref{prop:F-properties} verifies that $(\mathcal{E},\mathcal{F})$ is a strongly local, regular Dirichlet form. See also \cite[Proposition 2.2.8]{ChenFukushima}.

\begin{corollary}\label{cor:DF}
    The two variable form $(\mathcal{E},\mathcal{F})$ defined as above is a strongly local, regular Dirichlet form on $L^2(\mu)$.
\end{corollary}

\begin{proof}
    It is clear from the definition that $(\mathcal{E},\mathcal{F})$ is bilinear and non-negative definite. The fact that $(\mathcal{E},\mathcal{F})$ is closed follows from Proposition \ref{prop:F-properties}-(1) and the Markov property may easily be deduced from \cite[Theorem 4.4-(iii)]{Evans-Gariepy}.
    
    To check that $(\mathcal{E},\mathcal{F})$ is regular, note that $\mathcal{F} \subset \mathcal{C}$, meaning the first density is trivial. For the density $\mathcal{F} \cap \mathcal{C}$ in the uniform norm we argue as follows. Let $f \in \mathcal{C}$ and take a sequence of smooth functions $\varphi_n \in \mathcal{C}$ that approximates $f \circ \mathscr{U}^{-1}$ in $\norm{\,\cdot\,}_\infty$. Then $\varphi_n \circ \mathscr{U} \in \mathcal{F}$. Moreover, it follows from the uniform continuity of $\mathscr{U}$ that $\varphi_n \circ \mathscr{U} \to f$ in $\norm{\,\cdot\,}_\infty$.

    Finally strong locality follows from (2) of Proposition \ref{prop:F-properties}.
\end{proof}

\begin{remark}\label{rem:EM}
    Proposition \ref{prop:F-properties}-(5) implies that, when $p=2$, the Radon measure
    \[
    A \mapsto \frac{1}{2}\int_A \abs{\textup{D} f}^2\, \textup{d}\Lambda    
    \]
    is the \emph{energy measure} of $f \in \mathcal{F}$, see \cite[Section 3.2]{FOT}.
\end{remark}

\begin{remark}
    Consider a general $p \in (1,\infty)$ and define the Radon measures $\textup{d}\Gamma\langle f\rangle:=\abs{\textrm{D}f}^p\, \textup{d}\Lambda$.
    A similar argument as in the previous two proofs shows that the collection $([0,1],\abs{\,\cdot\,},\mu,\mathcal{E},\Gamma,\mathcal{F})$ is a $p$-Dirichlet space in the sense of \cite{EID2026}. The weak lower semicontinuity can be deduced, for instance, from a combination of Morrey inequality \cite[Theorem 4.10]{Evans-Gariepy} (see also the 6th exercise problem of \cite[Chapter 5]{evans}), Arzela-Ascoli theorem, and \cite[Theorem 5.2]{Evans-Gariepy}.
    Alternatively, suitable convexity properties of $\mathcal{E}$ can also be exploited \cite[Proposition 3.18]{KajinoShimizu}. The details are left to an interested reader.
\end{remark}

\section{Proof of Theorem \ref{thm-2}}\label{sec:4}
This section is dedicated to proving Theorem \ref{thm-2}, which we state here for general $p \in [1,\infty)$.
Throughout the section, we let $\Lambda,\mu,p$ and $(\mathcal{E},\mathcal{F})$ be as in Section \ref{sec:Framework}.

\begin{theorem}\label{thm:general-p}
    Let $\alpha \in (0,1]$. Then there exists a non-constant $\alpha$-Hölder continuous $f \in \mathcal{F}$ if and only if $\Lambda$ is generalized $\alpha$-Frostman.
\end{theorem}

\begin{proof}
    Let us begin with the forward implication. Let $f \in \mathcal{F}$ be a non-constant $\alpha$-Hölder continuous function, and fix a Borel measurable $\Lambda$-representative of $\textup{D}f$.
    Since $f$ is non-constant, the set $F := \{ x \in (0,1) : \textup{D}f \neq 0 \}$ has $\Lambda(F)>0$ by Proposition \ref{prop:F-properties}-(3). We let $E \subset F$ be the subset of $F$ consisting of the points $x \in F$ satisfying
    \[
         \textup{D}f(x) = \lim_{r \downarrow 0} \frac{1}{\Lambda(B(x,r))} \int_{B(x,r)} \textup{D}f \, \textup{d}\Lambda =  \lim_{r \downarrow 0} \frac{1}{\Lambda(B(x,r))} \int_{x-r}^{x+r} \textup{D}f \, \textup{d}\Lambda.
    \]
    The set $E$ has $\Lambda(E)>0$ by \cite[Corollary 2.14]{Mattila}. It remains to verify \eqref{eq:Frostman} for every $x\in E$. To this end, first compute
    \begin{align*}
        0 & < \abs{\textup{D}f(x)} = \lim_{r \downarrow 0} \left| \frac{1}{\Lambda(B(x,r))} \int_{x-r}^{x+r} \textup{D}f \,\textup{d}\Lambda \right| = \lim_{r \downarrow 0} \frac{\abs{f(x+r) - f(x-r)}}{\Lambda(B(x,r))}.
    \end{align*}
    The second equality follows from Proposition \ref{prop:F-properties}-(3).
    Then the $\alpha$-Hölder continuity of $f$ gives the estimate
    \[
        \lim_{r \downarrow 0} \frac{\abs{f(x+r) - f(x-r)}}{\Lambda(B(x,r))} \leq \liminf_{r \downarrow 0} \frac{2K r^\alpha}{\Lambda(B(x,r))}.
    \]
    A combination of the previous two displays yields
    \[
        \limsup_{r \downarrow 0} \frac{\Lambda(B(x,r))}{r^\alpha} \leq \frac{2K}{\abs{\textup{D}f(x)}} < \infty.
    \]
    Hence, $\Lambda$ is generalized $\alpha$-Frostman.

    Then we prove the backward implication. Let $E \subset [0,1]$ be as in Definition \ref{def:Frostman}. By taking a subset of $E$ if necessary, we may assume existence of $C,R > 0$ so that
    \begin{equation}\label{eq:Unif-Frostman}
        \frac{\Lambda(B(x,r))}{r^\alpha} \leq C
    \end{equation}
    for all $x \in E$ and $r \in (0,R)$, in other words, \eqref{eq:Frostman} holds in a uniform sense.
    Now, let $\mathds{1}_{E}$ be the characteristic function of $E$, and define
    \[
    f(x) := \int_0^x \mathds{1}_{E} \, \textup{d}\Lambda.
    \]
    It follows from Proposition \ref{prop:F-properties}-(4) that $f \in \mathcal{F}$.
    Since $f(0) = 0$ and $f(1) = \Lambda(E)>0$, it is non-constant.
    It remains to prove that $f$ is $\alpha$-Hölder.
    Note that it suffices show the Hölder continuity for points $x,y \in [0,1]$ with $\abs{x-y} \leq R/4$. Take such $x,y$ and assume $x \leq y$.
    If we would have $B(x,2\abs{x-y}) \cap E = \emptyset$ then $f$ is obviously constant on $B(x,2\abs{x-y})$. In particular, $\abs{f(x)-f(y)}=0$.
    Otherwise, by taking a point $z \in B(x,2\abs{x-y}) \cap E \neq \emptyset$ we get
    \[
    \abs{f(y) - f(x)} \leq \Lambda([x,y]) \leq \Lambda(B(z,4\abs{x-y})) \leq C 4^\alpha \abs{x-y}^\alpha.
    \]
    The last inequality follows from \eqref{eq:Unif-Frostman}.
\end{proof}

\begin{proof}[Proof of Theorem \ref{thm-2}]
The claim follows from Theorem \ref{thm:general-p} by taking $p=2$.
\end{proof}

Finally, we record a version of Corollary \ref{cor:delta_*} for the case of general $p\in[1,\infty)$. The claim follows immediately from Lemma \ref{lemma:upperHausdorff} and Theorem \ref{thm:general-p}. Recall that $\delta_*$ denotes the critical Hölder exponent as defined in the introduction.

\begin{corollary}\label{cor:general-p-delta_*}
    It holds that $\delta_*=\udimh\Lambda$.
\end{corollary}

\section{Proof of Theorem \ref{thm-1} and examples}\label{sec:examples}
This section provides the examples that, together with Theorem \ref{thm-2}, yield a proof of Theorem \ref{thm-1}.
All of the presented constructions are standard in fractal geometry. Nevertheless, since we expect that our audience is more inclined towards analysis, we describe them in detail.

\subsection{Mass distribution principle}
Let $E \subset [0,1]$ be a Borel set. Given $\alpha \in (0,1]$, the $\alpha$-\emph{Hausdorff content} of $E$ is defined as
\[
    \mathcal{H}_\infty^\alpha(E) := \inf \left\{ \sum_{i=1}^\infty r_i^\alpha \colon E \subset \bigcup_{i=1}^\infty B(x_i,r_i)   \right\}.
\]
The \emph{Hausdorff dimension} of $E$ is then the value
\[
    \dimh E := \sup \{ \alpha \in (0,1] : \mathcal{H}_\infty^\alpha(E) > 0 \},
\]
where $\sup \emptyset = 0$. A deep understanding of these definitions is not needed in this article because they are used only through the following \emph{mass distribution principle}.
For further details, we refer to \cite[Proposition 10.3]{Falconer}.

\begin{lemma}\label{lemma:Frostman-Lemma}
    Let $\nu$ be a Radon measure on $[0,1]$. If $E \subset [0,1]$ is a Borel set with $\nu([0,1]\setminus E)=0$ then $\udimh\nu\leq\dimh E$.
\end{lemma}

\subsection{Examples}

We provide the examples that complete the proof of Theorem \ref{thm-1}.

\begin{example}\label{prop:not-gen-delta-frostman}
    For any $\delta \in (0,1]$ there exists a non-atomic Radon probability measure $\nu$ on $[0,1]$ with full topological support that satisfies $\udimh\nu=\delta$.
    Moreover, $\nu$ can be chosen so that it is not generalized $\delta$-Frostman.
\end{example}

\begin{proof}
    We first take care of the case $\delta \in (0,1)$, for which the construction is simple; one can choose $\nu$ to be any fully supported self-similar measure satisfying the open set condition and $\udimh\nu=\delta$. For concreteness let us describe a construction of such a measure more precisely. Choose a parameter $0<\rho<1/2$, assign mass $\rho$ to the interval $[0,1/2]$ and mass $1-\rho$ to the interval $[1/2,1]$. Iterate the construction by subdividing each interval in the previous step of the construction into two dyadic subintervals, and assigning mass $\rho$ times the mass of the parent to the left subinterval and $1-\rho$ times the mass of the parent to the right one. Continuing this process \emph{ad infinitum}, and using Carathedory's extension theorem, one obtains a non-atomic Radon probability measure $\nu$ whose support is the full interval $[0,1]$. It is well known that $\nu$ satisfies
    \begin{equation*}
        \udimh \nu=\frac{\rho\log \rho+(1-\rho)\log(1-\rho)}{-\log 2}.
    \end{equation*}
    This follows from the law of large numbers, see for example \cite[Proposition 10.4]{Falconer}. Certainly, for a given $\delta \in (0,1)$, one can then choose the parameter $\rho$ in such a way that $\udimh \nu=\delta$. Moreover, it is a well known fact in fractal geometry that
    \begin{equation*}
        \limsup_{r\downarrow 0}\frac{\nu(B(x,r))}{r^{\delta}}=\infty,
    \end{equation*}
    for $\nu$-almost every $x \in [0,1]$, if $\delta \in (0,1)$.
    In other words, $\nu$ is not generalized $\delta$-Frostman. This fact is not too difficult to establish using the law of iterated logarithm, see for example \cite[Theorem 3.1]{BhouriHeurteaux09} and in particular the first displayed equation on their page 8.

    For the $\delta=1$ case, we have to be slightly more careful, since the above construction gives the Lebesgue measure which certainly is (generalized) $1$-Frostman. Instead, we choose a sequence of $\rho_n \in (0,1/2)$ tending to $1/2$, and let $\nu_n$ denote the measure constructed as above for the parameter $\rho_n$. Let $\nu$ be given by
    \begin{equation*}
        \nu :=\sum_{n=1}^{\infty}2^{-n}\nu_n.
    \end{equation*}
    It follows from Lemma \ref{lemma:upperHausdorff} and the discussion above that none of $\nu_n$ are generalized 1-Frostman. From this, it is easy to see that $\nu$ is not generalized 1-Frostman. It remains to show that $\udimh\nu =1$. First, by Lemma \ref{lemma:Frostman-Lemma}, $\udimh\nu$ is always bounded from above by the Hausdorff dimension of $[0,1]$.
    In other words, $\udimh\nu \leq 1$. Second, it is easy to check that $\udimh\nu \geq \udimh\nu_n$ for every $n \in \N$. Moreover, $\udimh\nu_n \to 1$ as $\rho_n \to 1/2$ by the first display in this proof. Hence $\udimh\nu \geq 1$. This concludes the case $\delta=1$, and completes the proof.
\end{proof}

\begin{example}\label{prop:0-dim-full-support}
    There exists a non-atomic Radon probability measure $\nu$ on $[0,1]$ with full topological support satisfying $\udimh\nu=0$.
\end{example}
\begin{proof}
    The construction is a simple modification of the one in Example \ref{prop:not-gen-delta-frostman}. Instead of assigning masses $\rho$ and $1-\rho$ to the subintervals, we fix a sequence $0<\rho_n\leq 1/2$ decreasing to $0$, satisfying $\prod_{n = 1}^{\infty} (1 - \rho_n)=0$, and at the $n$-th iteration, we assign a proportion of $\rho_n$ and $1-\rho_n$ of the masses of the parents to the resulting subintervals. Let $\nu$ be the resulting measure.
    
    Since the mass of an $k$-level dyadic interval can be bounded from above by $\prod_{n = 1}^k (1 - \rho_n)$, the choice of $\rho_n$ ensures that $\nu$ is non-atomic. It is also clear that $\nu$ has full topological support. Moreover, there is a well known formula for the upper Hausdorff dimension of $\nu$. For $k \in \N$ let $\mathcal{I}_k$ denote the set of $k$-level dyadic intervals. Then it holds that
    \[
        \udimh\nu= \limsup_{k\to\infty}\frac{-\sum_{I \in \mathcal{I}_k}\nu(I) \log \nu(I)}{k\log 2},
    \]
    see \cite[Theorem 3.1]{LiWu2011}.
    The right-hand side has a nice formula
    \[
        \sum_{I \in \mathcal{I}_k}\nu(I) \log \nu(I) = \sum_{n=1}^k(\rho_n \log \rho_n + (1-\rho_n)\log(1-\rho_n)).
    \]
    This follows from the iterative construction of $\nu$ and a simple induction on $k \in \N$.
    Now, since $\rho_n \to 0$, also $\rho_n \log \rho_n + (1-\rho_n)\log(1-\rho_n) \to 0$ as $n \to \infty$. By combining this with the above two displays, it is straightforward to conclude that $\udimh\nu=0$.
    \end{proof}

Recall the critical Hölder exponent $\delta_*$ from introduction. 

\begin{proof}[Proof of Theorem \ref{thm-1}]
    We first consider the case $\delta\in(0,1]$. Let $\Lambda = \nu$ be given by Example \ref{prop:not-gen-delta-frostman}, $\mu = \textup{d}x$, and $(\mathcal{E},\mathcal{F})$ be as in Definition \ref{def:p-EF} with $p=2$.
    Now, $(\mathcal{E},\mathcal{F})$ is a strongly local, regular Dirichlet form on $L^2(\mu)$ is proved in Corollary \ref{cor:DF}.
    The critical Hölder exponent $\delta_*$ of $([0,1],\abs{\,\cdot\,},\mu,\mathcal{E},\mathcal{F})$ is computed $\delta_* = \udimh \Lambda = \delta$ by a combination of Lemma \ref{lemma:upperHausdorff} and Theorem \ref{thm-2}.

    The case $\delta=0$ is handled similarly using Example \ref{prop:0-dim-full-support}.
\end{proof}

As it follows from Theorem \ref{thm-2} and Lemma \ref{lemma:upperHausdorff}, the critical Hölder exponent of the above examples is a strict supremum. The examples below show that this is not always the case; it may also be a maximum.
This gives an answer to (\textbf{Q2}) of Section \ref{sec:MainResult}.

\begin{example}\label{prop:is-gen-delta-frostman}
    For any $\delta \in (0,1]$ there exists a non-atomic  Radon probability measure $\nu$ on $[0,1]$ with full topological support so that $\udimh\nu = \delta$ and $\nu$ is generalized $\delta$-Frostman.
\end{example}
\begin{proof}
    For $\delta=1$, the Lebesgue measure is such an example, so it suffices to construct examples for $\delta \in (0,1)$. Our idea to this end is to find a dense $\delta$-dimensional fractal set and a suitable measure supported on it. Let us go through the details.

    Let $K \subset [0,1]$ be a compact subset for which the following holds. 
    There is $C \geq 1$ and a Radon probability measure $\mathcal{H}$ on $[0,1]$ satisfying $\mathcal{H}([0,1]\setminus K) = 0$ and
    \begin{equation}\label{eq:AR}
        C^{-1}r^\delta \leq \mathcal{H}(B(x,r)) \leq Cr^\delta
    \end{equation}
    for all $x \in K$ and $r \in (0,1)$.
    Such a $K$ may easily be found by suitably modifying the classical construction of the middle-third Cantor set and letting $\mathcal{H}$ be the restriction of the $\delta$-dimensional Hausdorff measure to $K$, see e.g. \cite[Corollary 3.3]{Falconer}.
    Then take a countable family of sets $K_n \subset [0,1]$ so that each $K_n$ is a rescaled and translated copy of $K$, and their union, denoted $Z \subset [0,1]$ is a dense subset.
    We also push-forward $\mathcal{H}$ along the same rescaling and translation, and denote by $\mathcal{H}_n$ the resulting Radon probability measure supported on $K_n$.

    Now, it follows from the upper bound of \eqref{eq:AR} that, for every $n \in \N$, there is $C_n \geq 1$ for which
    \[
        \mathcal{H}_n(B(x,r)) \leq C_n r^\delta
    \]
    holds for all $x \in [0,1]$ and $r > 0$.
    Finally, we define
    \[
        \nu := \left(\sum_{n \in \N} \frac{1}{2^n C_n } \right)^{-1}
        \sum_{n \in \N} \frac{1}{2^n C_n }  \mathcal{H}_n.
    \]
    This is clearly a non-atomic Radon probability measure and has full topological support by the lower bound of \eqref{eq:AR} and the density of $Z \subset [0,1]$.
    Moreover, it is (generalized) $\delta$-Frostman by the previous two displays.

    It remains to conclude that $\udimh\nu = \delta$. First, $\udimh\nu \geq \delta$ by the upper bound of \eqref{eq:AR}. Second, $\udimh\nu \leq \delta$ follows from Lemma \ref{lemma:Frostman-Lemma} and the fact that $\dimh Z = \delta$. The latter is a consequence of \eqref{eq:AR} and the countable stability of Hausdorff dimension, see \cite[Page 24]{Falconer}.
\end{proof}

    In the following example, we observe that the global Hölder regularity properties of $\mathscr{U}$ are not carried to functions in $\mathcal{F}$, in general.
    
    \begin{example}\label{ex:U-not-hölder}
        There exist a non-atomic Radon probability measure $\nu$ on $[0,1]$ with full topological support, such that $\udimh\nu = 1$, and the cumulative distribution function $U(x) := \nu([0,x])$ is not $\alpha$-Hölder continuous for any $\alpha \in (0,1]$.
    \end{example}

    \begin{proof}
        Let $\nu_s$ be the Radon measure $\nu$ constructed in Example \ref{prop:0-dim-full-support}, and let $\nu := 1/2(\nu_s + \textup{d}x)$. It follows from $2\nu \geq \textup{d}x$ and Lemma \ref{lemma:Frostman-Lemma} that $\udimh\nu = 1$. The other properties of $\nu$ in the claim are obvious. Now suppose that $U(x) := \nu([0,x])$ is $\alpha$-Hölder for some $\alpha \in (0,1]$. Then the function $x\mapsto \nu_s([0,x])=2U(x)-x$ is $\alpha$-Hölder, but this contradicts $\udimh\nu_s=0$.
    \end{proof}

    \subsection{Density properties}\label{sec:density}
    The remainder of the section is dedicated to studying (\textbf{Q3}) of Section \ref{sec:MainResult}.
    We first introduce helpful notation to simplify some of the statements. For $\alpha \in (0,1]$ we denote by $\mathcal{C}^\alpha \subset \mathcal{C}$ the space of $\alpha$-Hölder continuous functions $[0,1]\to\R$, and by $\mathcal{C}^+ \subset \mathcal{C}$ the space of Hölder continuous functions.
    
    Let us now begin with a simple lemma. 

    \begin{lemma}\label{lemma:subset}
        Let $\widetilde{\Lambda}$ be another Radon probability measure on $[0,1]$ with full topological support for which there is a constant $C \geq 1$ such that $\widetilde{\Lambda} \leq C \Lambda$. Let $(\widetilde{\mathcal{E}},\widetilde{\mathcal{F}})$ be another $p$-energy form defined by replacing $\Lambda$ with $\widetilde{\Lambda}$ in Definition \ref{def:p-EF}. Then $\widetilde{\mathcal{F}} \subset \mathcal{F}$.
    \end{lemma}

    \begin{proof}
        Let us denote the derivation $\widetilde{\textup{D}}f \in L^p(\widetilde{\Lambda})$ for $f \in \widetilde{\mathcal{F}}$.
        
        Let $f \in \widetilde{\mathcal{F}}$. Our objective is to show that $f \in \mathcal{F}$.
        By the absolute continuity $\widetilde{\Lambda} \ll \Lambda$, there is a non-negative $h\in L^1(\Lambda)$ satisfying $\textup{d}\widetilde{\Lambda}=h\textup{d}\Lambda$.
        It then follows from Proposition \ref{prop:F-properties}-(3) that
        \[
            f(x) = f(0) + \int_0^x \widetilde{\textup{D}}f \, \textup{d}\widetilde{\Lambda} = f(0) + \int_0^x \widetilde{\textup{D}}f \cdot h \, \textup{d}\Lambda.
        \]
        By Proposition \ref{prop:F-properties}-(4), it remains to show that $\widetilde{\textup{D}}f \cdot h \in L^p(\Lambda)$. For this, we use the hypothesis $\widetilde{\Lambda} \leq C\Lambda$. Indeed, it follows from the Lebesgue differentiation theorem that $0 \leq h(x) \leq C$ for $\widetilde{\Lambda}$-almost every $x \in [0,1]$. But then
        \begin{equation*}
             \int_0^1 |G|^p\textup{d}\Lambda=\int_0^1 |\widetilde{\textup{D}}f \cdot h|^p\textup{d}\Lambda= \int_0^1 h^{p-1}|\widetilde{\textup{D}}f|^p\textup{d}\widetilde{\Lambda}\leq C^{p-1}\int_0^1|\widetilde{\textup{D}}f|^p\textup{d}\widetilde{\Lambda}<\infty.
        \end{equation*}
    \end{proof}
    
    \begin{proposition}\label{prop:density}
        If $\Lambda=1/2(\nu_s+\textup{d}x)$ is as in Example \ref{ex:U-not-hölder} then the associated $p$-energy $(\mathcal{E},\mathcal{F})$ has the following property. The domain $\mathcal{F}$ contains all Lipschitz functions, but the subspace $\mathcal{C}^+ \cap \mathcal{F}$ is not dense in $(\mathcal{F},\norm{\, \cdot\,}_{\mathcal{F}})$.
    \end{proposition}
    \begin{proof}
    The fact that Lipschitz functions lie in $\mathcal{F}$ follows from Lemma \ref{lemma:subset} by taking $\widetilde{\Lambda} = \textup{d}x$. Indeed, in this case $\widetilde{\mathcal{F}} = W^{1,p}$ which contains the Lipschitz functions.

    Next let $f(x) := \nu_s([0,x])$. Again, $f \in \mathcal{F}$ by taking $\widetilde{\Lambda} = \nu_s$ in Lemma \ref{lemma:subset}. 
    The proof would be completed once we show that $f$ cannot be approximated by Hölder continuous functions in the Sobolev norm $\norm{\, \cdot\, }_{\mathcal{F}}$.
    To this end, suppose the contrary that it can be, and let $\{f_n\}_{n\in\N} \subset \mathcal{F}$ be such a sequence. First, observe that $\nu_s$ and $\textup{d}x$ are mutually singular, meaning there exists a Borel set $E \subset [0,1]$ with zero Lebesgue measure satisfying $\nu_s(E) = 1$.
    This follows from \cite[Proposition 10.3]{Falconer}, using $\dimh\nu_s<1$, and the fact that sets with Hausdorff dimension smaller than one have zero Lebesgue measure.
    Then take the sequence $\{ \tilde{f}_n \}_{n\in\N}$ given by
    \[
        \tilde{f}_n(x) := \int_0^x \textup{D}f_n \cdot \mathds{1}_E \, \textup{d}\Lambda
    \]
    for all $x \in [0,1]$.
    Note that, by (3) and (4) of Proposition \ref{prop:F-properties} $\abs{\tilde{f}_n(x) - \tilde{f}(y)} \leq \abs{f_n(x) - f_n(y)}$ for all $x,y \in [0,1]$. In particular, each $\tilde{f}_n$ is also Hölder continuous.
    
    By the convergence $\textup{D} f_n \to \textup{D} f$ in $L^p(\Lambda)$ and the equality $\textup{D} f = \textup{D} f \mathds{1}_E$, it follows that $\textup{D} f_n \mathds{1}_E \to \textup{D} f$ in $L^p(\Lambda)$. By Proposition \ref{prop:F-properties}-(3) we also have
    \[
        f(x) = \int_0^x \textup{D}f\, \textup{d}\Lambda
    \]
    for all $x \in [0,1]$. Hence $\tilde{f}_n(x) \to f(x)$ at every $x \in [0,1]$.
    
    Let $(\widetilde{\mathcal{E}},\widetilde{\mathcal{F}})$ be as in Lemma \ref{lemma:subset} with $\widetilde{\Lambda} = \nu_s$.
    By Proposition \ref{prop:F-properties}-(4), we have that $\{ \tilde{f}_n \}_{n\in\N} \subset \widetilde{\mathcal{F}}$.
    However, as noted above, each $\tilde{f}_n$ is Hölder continuous.
    Since $\udimh \nu_s = 0$, it follows from Corollary \ref{cor:general-p-delta_*} that $\{ \tilde{f}_n \}_{n\in\N}$ is a sequence of constant functions. But this and the pointwise convergence $\tilde{f}_n \to f$ would imply that $f$ is also constant, which is certainly not the case.
    \end{proof}
    Next we observe another interesting case where Hölder functions are dense in $\mathcal{F}$ while $\alpha$-Hölder functions are not for any fixed $\alpha\in(0,1]$.
    \begin{proposition}\label{prop:alphaHölder-notdense}
        There exists a non-atomic Radon probability measure $\Lambda$ with full topological support on $[0,1]$ such that the associated $p$-energy has the following property. The subspace $\mathcal{C}^+ \cap \mathcal{F}$ is dense in $(\mathcal{F},\norm{\,\cdot\,}_{\mathcal{F}})$ but the subspace $\mathcal{C}^\alpha \cap \mathcal{F}$ is not dense in $(\mathcal{F},\norm{\,\cdot\,}_{\mathcal{F}})$ for any $\alpha\in(0,1]$.
    \end{proposition}
    \begin{proof}
        We construct the measure $\Lambda$ as follows. Pick a sequence $\rho_n\downarrow 0$ with $\rho_1=1/2$ and let $\nu_n$ denote the self-similar measure constructed as in Example \ref{prop:not-gen-delta-frostman} with the probability parameter $\rho_n$. Let $\widetilde{\nu}_n$ denote the pushforward of $\nu_n$ under the unique affine orientation preserving map sending $[0,1]$ to $[1/(n+1),1/n]$ and define
        \begin{equation*}
            \Lambda=\sum_{n=1}^{\infty}2^{-n}\widetilde{\nu}_n.
        \end{equation*}
        It is clear that $\Lambda$ satisfies the required properties, so it remains to verify the density properties of $(\mathcal{E},\mathcal{F})$. We begin with the density of Hölder continuous functions.

        Note that for any level-$k$ dyadic interval $I$, it follows from the construction that
        \begin{equation*}
            \nu_n(I)\leq (1-\rho_n)^k,
        \end{equation*}
        where $\rho_n$ is the probability parameter used in the definition of $\nu_n$. By covering open balls with suitable dyadic intervals, it follows that $\nu_n$ is $\alpha_n$-Frostman for $\alpha_n := -\log(1-\rho_n)/\log 2$. Consequently, it follows that the restriction of $\Lambda$ to $[1/n,1]$ is also $\alpha_n$-Frostman.

        Now, let $f \in \mathcal{F}$ and assume without loss of generality that $f(0)=0$. Consider the functions
        \[
            f_n(x) := \int_0^x 
            \textup{D}f \cdot
            \mathds{1}_{ [1/n,1] \cap \{ \abs{\textup{D} f} \leq n \} } \, \textup{d}\Lambda.
        \]
        By Proposition \ref{prop:F-properties}-(4), it holds that $\{f_n\}_{n\in\N} \subset \mathcal{F}$. These are Hölder continuous because
        \[
            \abs{f_n(y)-f_n(x)}\leq n \int_x^y \mathds{1}_{[1/n,1]}\, \textup{d}\Lambda = n\Lambda([x,y] \cap [1/n,1]) \leq n C_n \abs{x-y}^{\alpha_n}.
        \]
        In the last inequality, we used the fact that $\Lambda$ restricted to $[1/n,1]$ is $\alpha_n$-Frostman. To conclude the density, we need to check that $f_n \to f$ in $(\mathcal{F},\norm{\,\cdot\,}_{\mathcal{F}})$.
        
        By Proposition \ref{prop:F-properties}-(4), it holds for each $n$ that $\textup{D}f_n = 
        \textup{D}f \cdot
        \mathds{1}_{[1/n,1] \cap \{ \abs{\textup{D} f} \leq n \} }$. In particular, $\textup{D}f_n \to \textup{D}f$ pointwise $\Lambda$-almost everywhere and $\abs{\textup{D}f_n} \leq \abs{\textup{D}f}$. The dominated convergence theorem then implies $\textup{D}f_n \to \textup{D}f$ in $L^p(\Lambda)$ as $n \to \infty$. Hence $\mathcal{E}(f-f_n) = \norm{\textup{D}f - \textup{D}f_n}_{L^p(\Lambda)} \to 0$ as $n\to\infty$. By using $f(0)=0$, it follows easily from Proposition \ref{prop:F-properties}-(3) that $f_n \to f$ in $\norm{\,\cdot\,}_{\infty}$, and hence $f_n \to f$ in $L^p(\textup{d}x)$. This completes the proof of the density.

        We move on to the non-density proof. Let $\alpha \in (0,1]$, and take a large enough $n$ so that $\udimh\widetilde{\nu}_n<\alpha$. We claim that the function
        \[
            f(x) := \int_0^x \mathds{1}_{[1/(n+1),1/n]} \, \textup{d}\Lambda = \int_0^x 2^{-n}\mathds{1}_{[1/(n+1),1/n]} \, \textup{d}\tilde{\nu}_n,
        \]
        which is contained in $\mathcal{F}$ by Proposition \ref{prop:F-properties}-(4),
        cannot be approximated by $\alpha$-Hölder functions.
        To this end, we assume that it can be and let $\{ f_k \}_{k \in \N} \subset \mathcal{F}$ be such an approximating sequence. By a similar argument as in Proposition \ref{prop:density}, the functions
        \[
            \tilde{f}_k(x) := \int_0^x \mathds{1}_{[1/(n+1),1/n]}\textup{D}f_k \textup{d}\Lambda = \int_0^x 2^{-n} \mathds{1}_{[1/(n+1),1/n]}\textup{D}f_k \textup{d}\tilde{\nu}_n
        \]
        form a sequence of $\alpha$-Hölder functions that converge to $f$ point-wise. Now, take any non-atomic Radon probability measure $\widetilde{\Lambda}$ with full topological support, $\widetilde{\nu}_n \leq C\widetilde{\Lambda}$ for some $C > 0$ and $\udimh \widetilde{\Lambda} < \alpha$. For instance, we may take $\widetilde{\Lambda} = 1/2(\tilde{\nu}_n + \nu_n)$.

        Let $(\widetilde{\mathcal{E}},\widetilde{\mathcal{F}})$ be the $p$-energy form where we replace $\Lambda$ with $\widetilde{\Lambda}$ in Definition \ref{def:p-EF}.
        A similar argument as in the proof of Lemma \ref{lemma:subset} shows that $\{\tilde{f}_k\}_{k\in \N}\subset \widetilde{\mathcal{F}}$. But since they are $\alpha$-Hölder, according to Corollary \ref{cor:general-p-delta_*}, they must be constant. The point-wise convergence $\tilde{f}_k \to f$ then implies that $f$ is also constant, which yields a contradiction.
        \end{proof}

        And lastly, we discuss examples where $\alpha$-Hölder functions are dense.

        \begin{proposition}
            For every $\delta \in (0,1]$ there is a non-atomic Radon probability measure $\Lambda$ with full topological support on $[0,1]$ and $\udimh \Lambda = \delta$ such that the associated $p$-energy has the following property. For every $\alpha \in (0,\delta)$ the subspace $\mathcal{C}^\alpha \cap \mathcal{F}$ is dense in $(\mathcal{F},\norm{\,\cdot\,}_\mathcal{F})$.
        \end{proposition}

        \begin{proof}
            The case $\delta = 1$ follows by taking $\Lambda=\textup{d}x$, so fix $\delta \in (0,1)$ and $\alpha \in (0,\delta)$.
            We let $\Lambda$ be the self-similar measure constructed in Example \ref{prop:not-gen-delta-frostman}. It is well-known that $\Lambda$ has the following property. For every $n \in \N$ there is $E_n \subset [0,1]$ and $C > 0$ such that $\Lambda([0,1]\setminus E_n)<1/n$ and
            \[
                \Lambda(B(x,r)) \leq Cr^\alpha
            \]
            for every $x \in E_n$ and $r>0$; this easily follows from \cite[Proposition 10.4]{Falconer}.

            Now, take an arbitrary $f \in \mathcal{F}$. Our objective is to find an approximating sequence of $\alpha$-Hölder functions. Without loss of generality, assume $f(0)=0$. For every $n \in \N$ define
            \[
                f_n(x) := \int_0^x \textup{D}f \cdot \mathds{1}_{E_n} \cdot \mathds{1}_{\{ \abs{\textup{D}f} \leq n\}}.
            \]
            By Proposition \ref{prop:F-properties}-(4), $\{f_n\}_{n\in\N}\subset \mathcal{F}$.
            Moreover, using the same argument as in the proof of the backward direction in Theorem \ref{thm:general-p}, it follows that
            \[
                \abs{f_n(y) - f_n(x)} \leq n \int_x^y \mathds{1}_{E_n} \leq n C4^\alpha \abs{x-y}^\alpha
            \]
            for all $0\leq x\leq y \leq 1$.
            This shows that $\{ f_n \}_{n \in \N} $ is a sequence of $\alpha$-Hölder functions, so it remains to conclude the convergence $f_n \to f$. First, it is a consequence of $\Lambda([0,1]\setminus E_n)\to 0$ and Proposition \ref{prop:F-properties}-(4) that $\textup{D}f_n \to \textup{D}f$ pointwise $\Lambda$-almost everywhere. Since we also have $\abs{\textup{D}f_n}\leq \abs{\textup{D}f}$ for each $n$, the convergence $f_n\to f$ in $(\mathcal{F},\norm{\,\cdot\,}_\mathcal{F})$ now follows from a similar argument as in the proof of the density in Proposition \ref{prop:alphaHölder-notdense}. This completes the proof.

        \end{proof}
        
\section{Analytic properties}\label{sec:analytic}
In this section, we aim to identify the required conditions for our constructions to exhibit finer analytic properties. For instance, when $p=2$ we are interested in heat kernel estimates.
Some informative examples are also provided.
These are unrelated to the Hölder continuity results but are definitely of independent interest.

Let us note that this is the only part of the article where the choice of $\mu$ plays a role. Moreover, the constructions in Section \ref{sec:InverseProblem} fail for $p=1$. For this reason, we assume $p\in(1,\infty)$ in the remainder of the article. The parameters $\Lambda$ and $\mu$, as well as the $p$-energy form $(\mathcal{E},\mathcal{F})$, are as in Section \ref{sec:Framework}.

\subsection{Poincar\'e and capacity}

We consider two analytic inequalities controlled by a suitable function $\Psi : (0,\infty) \to (0,\infty)$, as in Definition \ref{def:Psi}.
These are widely used in the general theory of Dirichlet forms and their $p$-variants, see \cite{GrigoryanTelcs,GHL15,ResistanceConjecture} and references therein for examples and applications.
The first one is a \emph{Poincar\'e inequality}, which in the present setting is the following. 
There is a constant $C > 0$ such that, for every $f \in \mathcal{F}$, $x \in [0,1]$ and $r > 0$ we have
\begin{equation}\label{PI}
    \tag*{$\textup{PI}_p(\Psi)$}
    \inf_{c \in \R} \int_{B(x,r)} \abs{f - c}^p \,\textup{d}\mu \leq C \Psi(r) \int_{B(x,r)} \abs{\textup{D}f}^p \, \textup{d}\Lambda. 
\end{equation}
The second one is an \emph{annular capacity upper estimate}, which refers to the following condition. There is $C > 0$ such that, for every $x \in [0,1]$ and $r > 0$, there is $\varphi \in \mathcal{F}$ with the properties $\varphi = 1$ on $B(x,r)$, $\varphi = 0$ on $[0,1] \setminus B(x,2r)$, and
\begin{equation}\label{Cap}\tag*{$\textup{Cap}_p(\Psi)$}
    \mathcal{E}(\varphi) \leq C \frac{\mu(B(x,r))}{\Psi(r)}.
\end{equation}
The following volume growth condition is often crucial in applications.
We say that a Radon measure $\nu$ on $[0,1]$ is \emph{doubling} if there is a constant $D > 0$ such that, for all $x \in [0,1]$ and $r > 0$, it holds that
\[
0 < \nu(B(x,2r)) \leq D\nu(B(x,r)).
\]

\begin{remark}\label{rem:HKE}
Recall the formula for the energy measures in Remark \ref{rem:EM}.
When $p=2$, the conjunction of the above three conditions is equivalent to \emph{sub-Gaussian heat kernel estimates} with space-time scaling $\Psi$. This follows from the very recent resolution of the Resistance conjecture\footnote{Using the earlier characterization that involve the cutoff Sobolev inequality is also feasible for our examples \cite{barlow2006stability}.} \cite{ResistanceConjecture}.
For general $p\in (1,\infty)$, these conditions imply the elliptic Harnack inequality of $p$-harmonic functions \cite{YangHarnack}.
Hence, by understanding when our construction satisfies them, we can provide new examples to the literature.
The examples can be extended further by taking Cartesian products in the sense of \cite{AEBProduct26,Strichartz}.
\end{remark}

We first consider these properties within a wider framework, without fixing the function $\Psi$.

\begin{proposition}\label{prop:PI}
    For every $f \in \mathcal{F}$, $x \in [0,1]$ and $r > 0$ it holds that
    \[
        \inf_{c \in \R} \int_{B(x,r)} \abs{f - c}^p \, \textup{d}\mu \leq \Lambda(B(x,r))^{p-1}\mu(B(x,r)) \int_{B(x,r)} \,\abs{\textup{D} f}^p \, \textup{d}\Lambda.
    \]
\end{proposition}

\begin{proof}
    For given $f \in \mathcal{F}$, $x \in [0,1]$ and $r > 0$ we choose the constant $c \in \R$ by
    \[
        c := \frac{1}{\mu(B(x,r))} \int_{B(x,r)} f \, \textup{d}\mu.
    \]
    The claimed inequality follows easily from the same arguments as in the proof of Lemma \ref{lemma:comparability}.
\end{proof}

\begin{proposition}\label{prop:Cap}
    Suppose that $\Lambda$ is doubling.
    Then there is $C > 0$ satisfying the following. For all $x \in [0,1]$ and $r > 0$ there is $\varphi \in \mathcal{F}$ such that $\varphi = 1$ in $B(x,r)$, $\varphi = 0$ in $[0,1] \setminus B(x,2r)$ and
    \begin{equation*}
    \mathcal{E}(\varphi) \leq C \Lambda(B(x,r))^{1-p}.
    \end{equation*}
\end{proposition}

\begin{proof}
    Let $x \in [0,1]$ and $r > 0$. For the sake of simplicity, let us assume that $(x-2r,x+2r)\subset [0,1]$. The general case follows from a similar argument.

    Define $\varphi \in \mathcal{F}$ according to
    \[
        \varphi(z) :=
        \begin{cases}
            0 & \text{ if } z \in [0,x-2r]\\
            
            \dfrac{\mathscr{U}(z)-\mathscr{U}(x-2r)}{\mathscr{U}(x-r)-\mathscr{U}(x-2r)} & \text{ if } z\in [x-2r,x-r] \\
            
            1 & \text{ if } z \in [x-r,x+r]\\
            \dfrac{\mathscr{U}(z)-\mathscr{U}(x+2r)}{\mathscr{U}(x+r)-\mathscr{U}(x+2r)} & \text{ if } z\in [x+r,x+2r] \\
            0 & \text{ if } z \in [x+2r,1].
        \end{cases}
    \]
    Since $\varphi \circ \mathscr{U}^{-1}$ is a piecewise linear continuous function, it is contained in $W^{1,p}$. Hence $\varphi \in \mathcal{F}$ by Definition \ref{def:p-EF}. The desired estimate follows easily by computing $\mathcal{E}(\varphi)$ and applying the doubling property of $\Lambda$.
\end{proof}

\subsection{Inverse problem}\label{sec:InverseProblem}
Propositions \ref{prop:PI} and \ref{prop:Cap} show that, in order to have \ref{PI} and \ref{Cap}, $\Lambda,\mu$ and $\Psi$ should be related to each other according to
\[
\Lambda(B(x,r))^{p-1}\mu(B(x,r)) \sim \Psi(r).    
\]
Here $\sim$ means a comparability up to a constant independent of $x$ and $r$.
This is a delicate detail because the left-hand side depends on $x$ whereas the right-hand side does not, see Remark \ref{rem:Fail} for an illustrative example.

In order to understand the obstructions of \ref{PI} and \ref{Cap}, we dedicate the remainder of the section to the \emph{inverse problem} of finding $\Lambda$ and $\mu$ from a given function $\Psi$ satisfying the following condition.

\begin{definition}\label{def:Psi}
    We say that an increasing homeomorphism $\Psi : (0,\infty) \to (0,\infty)$ is a \emph{scale function} if there are $1<\beta_L\leq \beta_U$ and $C \geq 1$ such that
    \begin{equation}\label{Psi}
        C^{-1} \left( \frac{R}{r} \right)^{\beta_L} \leq \frac{\Psi(R)}{\Psi(r)} \leq C \left( \frac{R}{r} \right)^{\beta_U}.
    \end{equation}
\end{definition}

A more general question that takes also into account the volume growth rate was recently studied in \cite{MuruganLaakso,YangLaakso25}. Our work is unable to handle this because the geometry of the real line is way too simple to this end. Nevertheless, we identify the following existence result.

\begin{theorem}\label{thm:inverse}
    Let $\Psi$ be a scale function which satisfies \eqref{Psi} for $p \leq \beta_L \leq \beta_U$.
    Then there exist doubling measures $\nu_1$ and $\nu_2$ on $[0,1]$ so that \ref{PI} and \ref{Cap} hold for $\Lambda = \nu_1$ and $\mu=\nu_2$.
\end{theorem}

\begin{remark}\label{rem:Psi}
    If $\Psi$ is a scale function and $\mu$ is a doubling measure so that \ref{PI} and \ref{Cap} both hold then $\Psi$ satisfies \eqref{Psi} for $p \leq \beta_L \leq \beta_U$. This is a general phenomenon for local $p$-energies on defined metric measure spaces, and the proof of \cite[Lemma 2.3]{MuruganLaakso} can be adapted to prove this.
\end{remark}

Our first step in the proof of Theorem \ref{thm:inverse} is to express $\Psi$ as a multiplicative cascade.
Judging from the discussion related to \cite[Lemma 5.6]{MuruganLaakso}, this seems to be a standard result. For the convenience of the reader, we provide the details.

\begin{lemma}\label{lemma:Cascade}
    Let $\Psi$ be a scale function, and let $p\leq \beta_L\leq \beta_U$ be as in \eqref{Psi}. Then there is $C \geq 1$ and $\mathbf{b} : \N \to \{\beta_U,\beta_L\}$ such that, for all $n \in \N $, we have
    \[
        C^{-1}\Psi(3^{-n}) \leq \prod_{j=1}^n 3^{-\mathbf{b}(j)} \leq C\Psi(3^{-n}).
    \]
\end{lemma}

\begin{proof}
    This proof constructs, recursively, the function $\textbf{b}$ defined on $\N$, and also an auxiliary function $\mathcal{B}$ defined on $\N \cup \{0\}$.
    We first define $\mathcal{B}(0) := 1$. Then let $i \in \N \cup \{0\}$ and assume that $\mathcal{B}(j)$ is defined for all $j \in \N \cup \{0\}$ with $j \leq i$, and that $\mathbf{b}(j)$ is defined for all $j \in \N$ with $j \leq i$. Then, if the condition
    \begin{equation}\label{eq:Stopping-time}
        \mathcal{B}(i) \leq \Psi(3^{-i})
    \end{equation}
    holds, we define $\mathbf{b}(i+1) := \beta_L$. Otherwise, $\mathbf{b}(i+1) := \beta_U$. In both cases, we define $\mathcal{B}(i+1) := \mathcal{B}(i)\cdot 3^{-\mathbf{b}(i+1)}$.

    We may assume that \eqref{eq:Stopping-time} holds for some $i \in \N$. Indeed, otherwise we use \eqref{Psi} to estimate
    \[
        3^{-i\beta_U} \leq
        C \frac{\Psi(3^{-i})}{\Psi(1)} \leq \frac{C}{\Psi(1)} \mathcal{B}(i) = \frac{C}{\Psi(1)} 3^{-i\beta_U},
    \]
    and the proof would be completed.
    By a similar argument, we may assume that \eqref{eq:Stopping-time} fails for some $i \in \N$.
    Hence, there is $N \in \N$ so that \eqref{eq:Stopping-time} holds for some $i \leq N$, and fails for some other $i \leq N$.

    We now prove the desired estimates. Let $n \in \N$ and assume without loss of generality that $n \geq N$.
    Then, if $i$ is the largest index such that $i \leq n$ and satisfies \eqref{eq:Stopping-time}, then 
    \[
        \mathcal{B}(n) = \mathcal{B}(i) \prod_{j=i+1}^n 3^{-\mathbf{b}(j)} \leq \Psi(3^{-i})\prod_{j=i+1}^n 3^{-\mathbf{b}(j)} \leq \Psi(3^{-i}) 3^{-(n-i)\beta_U} \leq C\Psi(3^{-n}).
    \]
    The last inequality follows from \eqref{Psi}.
    Similarly, if $i \leq n$ is the largest index for which \eqref{eq:Stopping-time} fails, then
    \[
        \mathcal{B}(n) = \mathcal{B}(i)\prod_{j=i+1}^n 3^{-\mathbf{b}(j)} \geq \Psi(3^{-i}) 3^{-(n-i)\beta_L} \geq C^{-1} \Psi(3^{-n}).
    \]
    This concludes the proof.
\end{proof}

In our second step, we introduce convenient probability densities.
This is the part where having $p$ strictly larger than $1$ is essential.

For $\beta \geq p$ we define $\rho_{1}(\beta),\rho_{2}(\beta),\rho_{3}(\beta)$ so that they solve the equation
\begin{equation}\label{eq:p-m}
    \begin{cases}
        \rho_1(\beta) = \rho_3(\beta) \text{ and } \rho_2(\beta) \in [1/3,1),\\
        \sum_{m=1}^3 \rho_{m}(\beta) = 1,\\
        \left(\sum_{m=1}^3 \rho_{m}(\beta)^{\frac{1}{1-p}}\right)^{1-p} = 3^{-\beta}.
    \end{cases}
\end{equation}
It follows from Hölder's inequality that $\beta \geq p$ is necessary for the well-posedness of \eqref{eq:p-m}. The sufficiency can be justified as follows. First, the values in the case $\beta = p$ are found by taking $\rho_m(p) = 1/3$ for each $m=1,2,3$. Then take $\rho_2 \in (1/3,1)$ and set $\rho_1 =  (1-\rho_2)/2 = \rho_3$. By sending $\rho_2 \to 1$, it follows that
\[
    \left(\sum_{m=1}^3 (\rho_{m})^{\frac{1}{1-p}}\right)^{1-p} \to 0.
\]
Hence, for every $\beta > p$, the values $\rho_m(\beta)$ can be found by a continuity argument.

Next, for $\beta \geq p$ we define $\rho^*_1(\beta),\rho^*_2(\beta),\rho^*_3(\beta)$ according to
\[
    \rho^*_{m}(\beta) := \left( \frac{3^{-\beta}}{\rho_m(\beta)} \right)^{\frac{1}{p-1}}.
\]
By \eqref{eq:p-m} it holds that
\begin{equation}\label{eq:q-m}
    \begin{cases}
        \rho^*_1(\beta),\rho^*_2(\beta),\rho^*_3(\beta) \in (0,1) \text{ and } \rho^*_1(\beta) = \rho^*_3(\beta),\\
        \sum_{m=1}^3 \rho^*_{m}(\beta) = 1,\\
        \rho^*_m(\beta)^{p-1} \cdot \rho_m(\beta) = 3^{-\beta} \text{ for all } m=1,2,3.
    \end{cases}
\end{equation}

\begin{proposition}\label{prop:Stopping-time}
    Let $\Psi$ be a scale function, and suppose that the constants in \eqref{Psi} satisfy $p \leq \beta_L \leq \beta_U$. Then there are doubling probability measures $\nu_1,\nu_2$ on $[0,1]$ and $C \geq 1$ such that, for all $x \in [0,1]$ and $r \in (0,1]$, we have
    \begin{equation}\label{eq:Comp}
        C^{-1}\Psi(r) \leq \nu_1(B(x,r))^{p-1} \nu_2(B(x,r)) \leq C \Psi(r).
    \end{equation}
\end{proposition}

\begin{proof}
    Let $\mathbf{b}$ be as in Lemma \ref{lemma:Cascade} and $\rho_m(\beta), \rho^*_m(\beta)$ as discussed above.
    We construct the measures $\nu_1$ and $\nu_2$ similarly as in Example \ref{prop:0-dim-full-support} as inhomogeneous Bernoulli measures, but instead of using the dyadic partition in the definition, we use the triadic one.
    
    To be slightly more precise, we construct $\nu_1$ by dividing mass to the first level triadic subintervals by assigning masses $\rho^*_1(\mathbf{b}(1))$, $\rho^*_2(\mathbf{b}(1))$ and $\rho^*_3(\mathbf{b}(1))$ to the left, middle and right intervals, respectively. We iterate this construction by distributing mass to level $n+1$ triadic intervals according to the weights $\rho^*_i(\mathbf{b}(n+1))$, that is, the left, middle and right descendants of a level $n$ triadic interval, get a proportion of $\rho^*_1(\mathbf{b}(n+1))$, $\rho^*_2(\mathbf{b}(n+1))$, and $\rho^*_3(\mathbf{b}(n+1))$ of the mass of the parent interval, respectively. This extends to a Radon probability measure on $[0,1]$ by Carathedory's extension theorem. The measure $\nu_2$ is constructed similarly but by using $\rho_i(\mathbf{b}(n))$ instead  $\rho^*_i(\mathbf{b}(n))$. The supports of $\nu_1$ and $\nu_2$ is the full unit interval because each triadic interval is assigned a positive measure. Finally, they are non-atomic by a similar argument as in Example \ref{prop:0-dim-full-support}.

   It remains to show that $\nu_1$ and $\nu_2$ are doubling measures and to prove \eqref{eq:Comp}.
   In the former, the arguments for $\nu_1$ and $\nu_2$ are identical so we only consider $\nu_1$.
   The fact that $\nu_1$ is doubling is a consequence of $\rho^*_1(\mathbf{b}(n))=\rho^*_3(\mathbf{b}(n))$ for all $n\in\N$. Let us go through the precise argument more carefully.
   
   Our first step is to prove that adjacent triadic intervals of same side length have comparable mass.
   To this end, take two constants $C,c > 0$ with the property
    \begin{equation*}
        c\leq \rho^*_i(\mathbf{b}(n))\leq C
    \end{equation*}
    for all $i=1,2,3$ and $n\in\N$.
    By the definition of $\rho^*_i(\mathbf{b}(n))$, we can choose $c$ and $C$ to depend only on $\beta_L,\beta_U$ and $p$.
    We prove by induction on $n \in \N \cup \{0\}$ that
    \[
    \nu(I) \leq (C/c)\nu(J)
    \]
    when $I$ and $J$ be adjacent $n$ level triadic intervals. The case $n=0$ is trivial because we would have $I=[0,1] =J$. Suppose that the claim holds for $n$, and take two adjacent level $n+1$ triadic intervals $I$ and $J$.
    Let $I_0$ and $J_0$ be the level $n$ that contain $I$ and $J$, respectively.
    If $I_0 = J_0$ then the inequality holds by construction. If $I_0 \neq J_0$, then $I_0$ and $J_0$ are adjacent. Moreover, neither $I$ or $J$ are the middle descendant of $I_0$ and $J_0$, respectively, because otherwise $I$ and $J$ cannot be adjacent. But then, according to the hypothesis $\rho_1^*(\mathbf{b}(n+1)) = \rho_3^*(\mathbf{b}(n+1))$, it holds that
    \[
        \nu_1(I) = \nu(I_0)\rho_1^*(\mathbf{b}(n+1)) \overset{\textup{(IH)}}{\leq} (C/c)\nu(J_0)\rho_1^*(\mathbf{b}(n+1)) = (C/c)\nu(J).
    \]
    Here $\textup{(IH)}$ stands for the use of induction hypothesis. This completes the first step.
    
    Next, take any open ball $B(x,r)$ for $x \in [0,1]$ and $r>0$, and let $L>0$ be the distance between its endpoints.
    Also let $n \in \N \cup \{0\}$ such that $3^{-n-1} < L \leq 3^{-n}$. Then $B(x,r)$ contains a level $n+2$ triadic interval $I$, and $B(x,2r)$ can be covered by 60 level $n+2$ triadic intervals so that their union is an interval.
    Using the estimate obtained in the previous step, it now follows that
    \[
        \nu_1(B(x,2r)) \leq 60 \left( \frac{C}{c} \right)^{60} \nu_1(I) \leq 60 \left( \frac{C}{c} \right)^{60}\nu(B(x,r)).
    \]
    This completes the proof of the doubling property of $\nu_1$. The doubling property of $\nu_2$ follows from an identical argument.

    We proceed to verifying \eqref{eq:Comp}. By the last row of \eqref{eq:q-m}, it holds for any triadic interval $I$ of level $n$ that
   \begin{equation*}
       \nu_1(I)^{p-1}\nu_2(I) = \prod_{j = 1}^n 3^{-\mathbf{b}(j)}.
   \end{equation*}
   By Lemma \ref{lemma:Cascade}, the right-hand side is comparable to $\Psi(3^{-n})$. By combining the previous equality with the doubling properties of $\nu_1,\nu_2$ and \eqref{Psi}, we obtain the desired conclusion.
\end{proof}

The following remark shows that the condition \eqref{eq:Comp} is somewhat delicate. We shall streamline the discussion by omitting some inessential multiplicative constants when writing $\sim$ and $\lesssim$.

\begin{remark}\label{rem:Fail}
    In general, if $\nu_1$ is only a doubling measure, then there might not exist a doubling measure $\nu_2$ and a scale function $\Psi$ such that \eqref{eq:Comp} holds.
    A simple counterexample is given by the measure $\textup{d}\nu_1 = x\, \textup{d}x$ and $p = 2$.
    It is an easy exercise to check that $\nu_1$ is doubling.
    We assume that $\nu_2$ and $\Psi$ exist and derive a contradiction.
    Let $N \in \N$ be an even number. It holds by \eqref{eq:Comp} that
    \[
        \nu_2([0,1/2]) = \sum_{k = 0}^{N/2-1} \nu_2\left( \left[ \frac{k}{N},\frac{k+1}{N} \right] \right) \sim \Psi(N^{-1}/2)\sum_{k=0}^{N/2-1} \nu_1\left( \left[ \frac{k}{N},\frac{k+1}{N} \right] \right)^{-1}.
    \]
    Also it holds that
    \[
        \nu_2([1/2,1])\sim \Psi(N^{-1}/2)\sum_{k=N/2}^{N-1} \nu_1\left( \left[ \frac{k}{N},\frac{k+1}{N} \right] \right)^{-1}.
    \]
    Moreover, an integration gives $\nu_1([k/N,(k+1)/N]) = (2k+1)/N$.
    These properties combined yield for all large enough $N$ that
    \[
        \Psi(N^{-1}/2)N^2 \log N \lesssim 1 \lesssim \Psi(N^{-1}/2)N^2. 
    \]
    These inequalities obviously fail when $N$ is large enough. Note that the above argument is valid whenever $\nu_1$ and $\nu_2$ are non-atomic Radon measure on $[0,1]$ with full topological support and $\Psi(r) > 0$ for all $r > 0$.
\end{remark}

\begin{proof}[Proof of Theorem \ref{thm:inverse}]
    A combination of Propositions \ref{prop:PI}, \ref{prop:Cap} and \ref{prop:Stopping-time} proves the claim.
\end{proof}

\bibliographystyle{acm}
\bibliography{DF}

\end{document}